\documentclass[11pt]{article}
\usepackage[english]{babel}
\usepackage[document]{ragged2e}
\usepackage[skip=10pt plus1pt, indent=40pt]{parskip}
\usepackage{longtable}
\usepackage[top=0.75in, bottom=0.5in, left=0.8in, right=0.8in,includefoot,heightrounded]{geometry}
\usepackage[utf8]{inputenc}
\usepackage[procnames]{listings}
\usepackage{pdflscape}
\usepackage{upgreek}
\usepackage{fancyhdr}

\usepackage{mathtools}
\usepackage{dsfont}
\usepackage{amssymb, amsmath, amsbsy,amsthm,amscd}
\usepackage{bigints}
\usepackage{stmaryrd}
\usepackage{cases}
\usepackage{accents}

\newtheorem{defn0}{Definition}[section]
\newtheorem{prop0}[defn0]{Proposition}
\newtheorem{thm0}[defn0]{Theorem}
\newtheorem{lemma0}[defn0]{Lemma}
\newtheorem{corollary0}[defn0]{Corollary}
\newtheorem{example0}[defn0]{Example}
\newtheorem{remark0}[defn0]{Remark}
\newtheorem{conjecture0}[defn0]{Conjecture}
\newtheorem{ansatz0}[defn0]{Ansatz}
\newtheorem{notation0}[defn0]{Notation}

\newenvironment{definition}{ \begin{defn0}}{\end{defn0}}
\newenvironment{proposition}{\bigskip \begin{prop0}}{\end{prop0}}
\newenvironment{theorem}{\bigskip \begin{thm0}}{\end{thm0}}
\newenvironment{lemma}{\bigskip \begin{lemma0}}{\end{lemma0}}

\newenvironment{remark}{ \bigskip \begin{remark0}}{\end{remark0}}

\newcommand{\secref}[1]{Section~\ref{#1}}

\mathtoolsset{showonlyrefs}

\numberwithin{equation}{section}
\numberwithin{figure}{section}

\usepackage{overpic}
\usepackage{subcaption}
\usepackage{graphicx}

\usepackage{orcidlink}
\usepackage{hyperref}
\definecolor{Cerulean}{RGB}{29,172,214}
\definecolor{JungleGreen}{cmyk}{0.99,0,0.52,0}
\hypersetup{
    colorlinks=true,
    linkcolor = black,
    citecolor = Cerulean,
    filecolor=magenta,      
    urlcolor=blue,
}

\usepackage{authblk}

\makeatletter

\newcommand{\tpitchfork}{%
  \vbox{
    \baselineskip\z@skip
    \lineskip-.52ex
    \lineskiplimit\maxdimen
    \m@th
    \ialign{##\crcr\hidewidth\smash{$-$}\hidewidth\crcr$\pitchfork$\crcr}
  }%
}
\makeatother

\usepackage{biblatex}
\title{Breaking of invariant curves: from the Fermi-Ulam map to the breathing circle billiard}

\author[1]{Jos\'e Lamas\thanks{Corresponding author. E-mail: \texttt{lamasrodriguezjose@dlut.edu.cn}} \orcidlink{0000-0002-1809-1823}}
\author[2]{Stefano Mar\`o\thanks{E-mail: \texttt{marostefano@uniovi.es}} \orcidlink{0000-0003-1194-1052}}
\affil[1]{School of Mathematical Sciences, Dalian University of Technology, P.R. China}
\affil[2]{Facultad de Ciencias, Universidad de Oviedo, Oviedo, Spain}

\date{}
\begin{document}
\justifying
\maketitle

\begin{abstract}
We consider the breathing circle billiard, in which a point particle moves freely inside a disk. The radius varies periodically in time, with elastic reflections at the moving boundary. In this system the angular momentum is preserved, and fixing its value $c$ reduces the dynamics to a two-dimensional exact symplectic map on a cylinder. In the high-energy regime this map is a twist map generated by a diagonally periodic generating function $h_c$. 

We study the small angular momentum regime as a perturbation of the limiting case $c=0$, which corresponds to the Fermi-Ulam dynamics along a diameter. Using this perturbative structure and a quantitative version of Mather's converse-KAM criterion, we exclude invariant Lipschitz graphs for suitable rotation numbers. Combined with Aubry-Mather theory and Forni's theorem, this yields positive topological entropy for sufficiently small $c\geq0$. Our result improves previous similar results obtained via the standard Mather's converse-KAM criterion by giving a sharper quantitative threshold for the destruction of invariant curves.
\end{abstract}

\section{Introduction}\label{sec:intro}

A mathematical billiard with moving boundary is a region of the plane whose boundary changes with time. A point particle moves freely inside this region and reflects elastically at the moving boundary. Since the boundary can exchange energy with the particle, the kinetic energy is not conserved at reflections. A natural question is therefore whether the successive bounces can produce instability phenomena, such as chaotic motion or unbounded growth of energy.

This question traces back to the Fermi-Ulam model~\cite{Fermi1949, Ulam1961}, where one studies the motion of a particle between two periodically moving walls. The answer depends strongly on the regularity of the motion of the wall. When the motion is sufficiently smooth, invariant curves appear at high energy and prevent acceleration~\cite{LaederichLev1991}. On the other hand, when the motion has low regularity, one can construct trajectories whose energy grows unbounded~\cite{Zharnitsky1998} (see also~\cite{Maro2014, Maro2013} for related results in impact systems).

Time-dependent billiards may be viewed as natural generalizations of the Fermi-Ulam model~\cite{GelfreichRomKedarTuraev2012,KoillerMarkarianKamphorstCarvalho1995}. In this broader setting, the geometry of the moving boundary plays an essential role. For instance, in time-dependent domains close to ellipses, one can find accelerating orbits together with the phenomena of splitting of separatrices~\cite{DettmannFainTuraev2018}. The situation is different for the breathing circle billiard. In this case, sufficiently smooth motions of the boundary produce invariant curves at high energy and therefore bounded energy growth~\cite{OliffsonKamphorstCarvalho1999}. This regular behavior, however, does not rule out the coexistence of chaotic dynamics, as shown in~\cite{BonannoMaro2022}. More recently, stability, global dynamics and Aubry-Mather sets for the breathing circle billiard have been further investigated in~\cite{CaoMaYuTanRenQi2022, ZhangLiLiu2025,ZhangXieLiCaoGrebogi2022}. 

Among time-dependent billiards, the breathing circle billiard is one of the simplest models in which the effect of a moving boundary can be studied beyond the one-dimensional setting. It is defined by
\[\mathcal D_t := \{x\in \mathds R^2\colon |x|<R(t)\},\]
where $R$ is strictly positive and $1$-periodic. The radial motion of the boundary allows the particle to exchange energy with the billiard table, while angular momentum is preserved due to the circular symmetry. Thus, after fixing the angular momentum $c \ge 0$, the dynamics reduces to a two-dimensional map of the cylinder.

This reduction makes the breathing circle billiard closely related to exact symplectic twist maps. In suitable high-energy coordinates, the reduced map admits a generating function which satisfies a twist condition (see, for instance,~\cite{BonannoMaro2022, KunzeOrtega2011}). The case $c=0$ is degenerate from the billiard point of view but remains dynamically meaningful: the particle moves along a diameter and the reduced dynamics coincides with the Fermi-Ulam map~\cite{Ulam1961}. Therefore the breathing circle billiard can be viewed, for small angular momentum, as a natural two-dimensional continuation of the Fermi-Ulam model.

The exact symplectic twist structure provides a natural variational framework for the study of these maps. Indeed, exact symplectic twist maps are described by generating functions, and their orbits correspond to stationary points of an action functional. Within this setting, Aubry-Mather theory gives minimizing invariant sets with prescribed rotation number~\cite{Bangert1988, Mather1991, MatherForni1994}. For irrational rotation numbers, such a set is either an invariant Lipschitz graph or a Cantor set. This alternative is central in what follows: when invariant graphs are excluded, Forni's theorem yields invariant ergodic probability measures with positive metric entropy, and hence also positive topological entropy~\cite{Forni1996} (see also~\cite{Angenent1992, Angenent1990} for related results).

The exclusion of invariant graphs belongs to converse-KAM theory~\cite{Haro1999, MacKayMeissStark1989,MacKayPercival1985} and the conditions under which it is guaranteed strongly depends on the involved converse-KAM criterion. In this paper, we focus on its variational form~\cite{Bangert1988, Mather1982}. More precisely, given an invariant curve $\Gamma$ and its associated Birkhoff map $\varphi(x)$ (see Lemma \ref{lemma: birk}), Mather proved that
\begin{equation}\label{mather_breaking}
a(x):=h_{22}(\varphi^{-1}(x),x)+h_{11}(x,\varphi(x))>0 \qquad \forall x\in\mathds R.
\end{equation}
This criterion has been used, among other contexts, in impact models and in the breathing circle billiard in~\cite{BonannoMaro2022,Maro2020}. 

In this paper, we use a quantitative refinement of \eqref{mather_breaking} due to \cite{Maro2020}, where an explicit function $L(x)$ is obtained such that
\begin{equation}\label{mather_breaking_maro}
a(x)>L(x)>0 \qquad \forall x\in\mathds R.
\end{equation}
We shall use this sharper estimate to improve the result of~\cite{BonannoMaro2022}. As already mentioned, the result concerning the breathing circle billiard is obtained as a perturbation of the Fermi-Ulam case. Since both models are of intrinsic interest, we state the results separately: Theorem \ref{thm: chaos FU} is devoted to the Fermi-Ulam map, while Theorem \ref{thm: chaos} concerns the breathing circle billiard map.

The paper is organized as follows. The formal statement of our main results and a comparison with \cite{BonannoMaro2022} is given in Section \ref{subsec: main results} below.  In~\secref{sec: generating function}, we consider the generating function of the Fermi-Ulam map and the ``reduced'' breathing circle billiard map with angular momentum $c > 0$. The connection between the two is given in Lemma \ref{lem: C2 convergence}. The proof of the main results is given in ~\secref{sec: chaos}. We first prove Theorem~\ref{thm: chaos FU} and, subsequently, we show how Lemma \ref{lem: C2 convergence} allows us to extend the result to the breathing circle billiard, proving Theorem~\ref{thm: chaos}.

\subsection{Main results}\label{subsec: main results}

Let $R\colon\mathds R\to\mathds R$ be a $C^2$, strictly positive and $1$-periodic function. For each $t\in\mathds R$, we consider the disk
\begin{equation}\label{eq: admissible domain}
\mathcal D_t:=\left\{x\in\mathds R^2: |x|<R(t)\right\}.
\end{equation}
A unit-mass particle moves freely in $\mathcal D_t$ and reflects elastically on the moving boundary. Fixing an angular momentum $c\geq0$, we denote by $\mathcal P_c$ the corresponding ``reduced'' breathing circle billiard map. For the case of null angular momentum the map reduces to the Fermi-Ulam map, denoted by $\mathcal P_0$.  The precise formulations of $\mathcal P_c$ and $\mathcal P_0$ are recalled in Section~\ref{sec: generating function}.

\begin{definition}\label{def: function R}
Let $R\in C^2(\mathds R)$ be strictly positive and $1$-periodic. Denote by $\|\cdot\|$ the $\sup$-norm. Set
\[\underline R := \min_{t\in\mathds R} R(t),\qquad\overline R := \max_{t\in\mathds R} R(t).
\]
We define
\begin{equation}\label{eq: sigma}
\sigma_B:=\min\left\{\frac{\underline R}{2\|\dot R\|},\frac{2\, \sqrt{1+\sqrt{1-\varepsilon^2}}\, \underline{R}}{\sqrt{\| \frac{d^2}{dt^2} R^2 \|}} \right\}, \qquad \sigma_0:=\frac{\underline R}{\|\dot R\|}
\end{equation}
with the convention that $\underline R/\|\dot R\|=+\infty$ if $\|\dot R\|=0$ and $\varepsilon\in(0,1)$ is a fixed parameter.
%
%

We say that $R$ belongs to the class $\mathcal R_B$ or $\mathcal R_0$ if $\sigma_B>2$ or $\sigma_0>2$, respectively. 

We say that $R$ belongs to the class $\widetilde{\mathcal R}_B$ or $\widetilde{\mathcal R}_0$ if $R\in\mathcal R_B$ or $R\in\mathcal R_0$ respectively and there exist $\bar t\in\mathds R$ and a constant $\kappa_R>0$ such that
\begin{equation}\label{eq: kappaR}
R(\bar t)=\overline R,\qquad \ddot R(\bar t)<-\kappa_R,
\end{equation}
and 
    \begin{equation}\label{eq: sharpened criterion}
    \kappa_R>\inf\frac{\omega+1}{2(\overline R+\underline R)}\left(\frac{8\overline R^2}{(\omega-1)^3}-A_{\mathrm{low}}(\omega,R)\right)>0,
    \end{equation}
    where the infimum is taken for $\omega\in(1,\sigma_B-1)$ or $\omega\in(1,\sigma_0-1)$, respectively. For $R\in\mathcal R_0$ and $\omega\in(1,\sigma_0-1)$ the function $\omega\mapsto A_{\mathrm{low}}(\omega,R)$ is defined in~\eqref{eq: Alow 0} and is continuous, strictly positive and vanishes as $\omega\to 0^+$. \\ 
        Note that the inclusions $\mathcal R_B\subset \mathcal R_0$ and $\widetilde{\mathcal R}_B\subset \widetilde{\mathcal R}_0$ come from the definition.   
    
%
\end{definition}

We state the main results of this paper.

\begin{theorem}\label{thm: chaos FU}
If $R\in\widetilde{\mathcal R}_0$, then the map $\mathcal P_0$ has positive topological entropy.
\end{theorem}

\begin{theorem}\label{thm: chaos}
If $R\in\widetilde{\mathcal R}_B$, then there exists $c_0>0$ such that, for every $c\in(0,c_0)$, the map $\mathcal P_c$ has positive topological entropy.
\end{theorem}

\begin{remark}\label{rem: cond2}
    It follows from the proof of Lemma~\ref{lem: Aup0} that the assumption~\eqref{eq: kappaR} may be replaced by
    \[\dot R(\bar t)=0,\qquad \ddot R(\bar t)<-\kappa_R,\]
    provided that in the definition of the threshold~\eqref{eq: sharpened criterion} one may replace the denominator $2(\underline R+\overline R)$ by $4\underline R$.
    \end{remark}
The presence of the strictly positive term $A_{\mathrm{low}}$ is the main novelty of this paper. It comes from applying the converse-KAM criterion in~\cite{Maro2020} instead of the criterion in~\cite{Mather1982}. This is detailed in the following remark.
\begin{remark}
The relevant condition in the result in~\cite{BonannoMaro2022} comes from Mather's converse-KAM criterion in~\cite{Mather1982}. With our notation this reads 
\begin{equation}\label{eq: MB cond}
    \kappa_R>\frac{2\overline R^2}{\underline R\,\sigma_B^2},
\qquad
\sigma_B>4.
\end{equation}    
Following the same approach, our threshold \eqref{eq: sharpened criterion} becomes
\begin{equation}\label{eq: Bangert-Mather criterion}
\kappa_R>\inf_{\omega\in(1,\sigma_B-1)}\frac{\omega+1}{2(\overline R+\underline R)}\cdot\frac{8\overline R^2}{(\omega-1)^3},
\end{equation}
in which the term $A_{\mathrm{low}}(\omega,R)$ disappear. From Remark \ref{rem: cond2} and the fact that $\omega\mapsto(\omega+1)/(\omega-1)^3$ is decreasing for $\omega>1$ we recover the threshold \eqref{eq: MB cond}. \\
Moreover, by the continuity of $ \omega\mapsto A_{\mathrm{low}}(\omega,R)$, condition \eqref{eq: sharpened criterion} is still satisfied when we choose $\omega=\sigma_B-1$. This leads to the condition
\[\kappa_R>\frac{2\overline R^2}{\underline R\,\sigma_B^2}-\frac{\sigma_B}{4\underline R}A_{\mathrm{low}}(\sigma_B-1,R),\qquad\sigma_B>2,\]
that improves \eqref{eq: MB cond}.

    %
    %
     
Note that already in the case $A_{\mathrm{low}}=0$ our criterion is sharper. We only assume $\sigma_B>2$ and do not require the additional assumptions~\textup{(i)} and~\textup{(ii)} in~\cite[Definition~2.2]{BonannoMaro2022}. This improvement comes from working in the configuration space instead of the phase space.  
\end{remark}

\section{The generating functions of the breathing billiard and the Fermi-Ulam maps}\label{sec: generating function}

Following~\cite{BonannoMaro2022}, after reduction by rotational symmetry, the dynamics of the breathing circle with fixed angular momentum $c>0$ is described by an exact symplectic twist map of the cylinder $\mathcal P_c$. This allows us to describe the orbits as solutions of a difference equation involving a generating function. More precisely, we write
\[\tau:=t_1-t_0,\qquad R_0:=R(t_0),\qquad R_1:=R(t_1)\]
and define the set
%
%
%
%
%
\begin{equation}\label{eq: Omega}
\Omega_B:=\left\{(t_0,t_1)\in\mathds R^2:0<t_1-t_0<\sigma_B\right\},
\end{equation}
where $\sigma_B$ is defined in~\eqref{eq: sigma}. 
On $\Omega_B$ we define the function
\begin{equation}\label{eq: generating function c}
h_c(t_0,t_1):=\frac{R_0^2+R_1^2+2\sqrt{R_0^2R_1^2-c^2\tau^2}}{2\tau}+c\arctan\left(\frac{c\tau}{\sqrt{R_0^2R_1^2-c^2\tau^2}}\right), \qquad \quad c\in \left(0,\varepsilon\, \frac{\underline{R}^2}{\sigma_B}\right).
\end{equation}
This is the generating function for the reduced breathing circle billiard map on $\Omega_B$ (see~\cite[Section 4]{BonannoMaro2022}). 
More precisely
\begin{equation}\label{eq: diagonal periodic h}
h_c(t_0+1,t_1+1)=h_c(t_0,t_1),\qquad \partial_{t_1t_0}h_c<0
\end{equation}
on $\Omega_B$, and the reduced map $\mathcal P_c$ is defined implicitly by
\begin{equation}\label{eq: implicit relation K-h}
K_0=\partial_1h_c(t_0,t_1),\qquad K_1=-\partial_2h_c(t_0,t_1),
\end{equation}
so that
\begin{equation}\label{eq: billiard map}
\mathcal P_c(\bar t_0,K_0)=(\bar t_1,K_1), \qquad \bar t_i=t_i+\mathds Z.
\end{equation}
In these coordinates, $K_0\,dt_0-K_1\,dt_1=dh_c$. Thus, the map $\mathcal P_c$ is exact symplectic and twist in $\Sigma_B:=\mathds T\times \mathds(\sigma_B^*,+\infty)$ with $\sigma_B^*=\max_{t}\partial_1h_c(t,t+\sigma_B)$. We summarize this in the following proposition.

\begin{proposition}\label{prop: DEL orbit correspondence}
Let $(t_n)_{n\in\mathds Z}$ satisfy
\begin{equation}\label{eq: DEL}
\partial_1h_c(t_n,t_{n+1})+\partial_2h_c(t_{n-1},t_n)=0
\end{equation}
and assume that $(t_n,t_{n+1})\in\Omega_B$ for every $n\in\mathds Z$. Then $(t_n)_{n\in\mathds Z}$ is in one-to-one correspondence with an orbit of the reduced map $\mathcal P_c$ generated by $h_c$.
\end{proposition}

The case with null angular momentum ($c=0$) corresponds to the motion on the diameter. It is equivalent to the Fermi-Ulam problem~\cite{Ulam1961}. Substituting formally $c=0$ in~\eqref{eq: generating function c} we get
\begin{equation}\label{eq: generating function FU}
h_0(t_0,t_1)=\frac{(R_0+R_1)^2}{2(t_1-t_0)},    
\end{equation}
that has been proved to be the generating function of the Fermi-Ulam map in~\cite[Section~4]{MaroOrtega2026} on the set 
\begin{equation}\label{Omega0}
\Omega_0:=\left\{(t_0,t_1)\in\mathds R^2:0<t_1-t_0<\sigma_0\right\},
\end{equation}
where $\sigma_0$ is defined in~\eqref{eq: sigma}. In particular, 
\begin{equation}\label{eq: h0 twist}
h_0(t_0+1,t_1+1)=h_0(t_0,t_1),\quad -\partial_{12}h_0(t_0,t_1)=\frac{1}{\tau}\left(\frac{R_0+R_1}{\tau}+\dot R_0\right)\left(\frac{R_0+R_1}{\tau}-\dot R_1\right)>0
\end{equation}
on $\Omega_0$. Analogously to the previous case, we denote by $\mathcal P_0$ the associated map in the cylinder. Defining 
\[\sigma_0^*:=\max_{t}\partial_1h_0(t,t+\sigma_0),\]
the map $\mathcal P_0$ is an exact symplectic twist map in $\Sigma_0:=\mathds T\times \mathds(\sigma_0^*,+\infty)$. We state this result in the following proposition.

\begin{proposition}\label{prop: DEL orbit correspondence0}
Let $(t_n)_{n\in\mathds Z}$ satisfy
\begin{equation}\label{eq: DEL0}
\partial_1h_0(t_n,t_{n+1})+\partial_2h_0(t_{n-1},t_n)=0
\end{equation}
and assume that $(t_n,t_{n+1})\in\Omega_0$ for every $n\in\mathds Z$. Then $(t_n)_{n\in\mathds Z}$ is in one-to-one correspondence with an orbit of the map $\mathcal P_0$ generated by $h_0$.
\end{proposition}

The relation between these two generating functions is given by the following lemma.

\begin{lemma}\label{lem: C2 convergence}
Let $\Omega_B$ and $\Omega_0$ be the sets defined in~\eqref{eq: Omega} and~\eqref{Omega0}, respectively. Then $\Omega_B\subset\Omega_0$ and the following expansion holds uniformly in $\Omega_B$ as $c\to 0^+$:
%
%
%
\[h_c=h_0+\mathcal O_{C^2}(c),
\]
where $h_c$ and $h_0$ are defined in~\eqref{eq: generating function c} and~\eqref{eq: generating function FU}, respectively. Moreover, uniformly in $\Sigma_B$, as $c\to 0^+$ we have
\[
\mathcal P_c=\mathcal P_0 + \mathcal O_{C^1}(c).
\]
\end{lemma}

\begin{proof}
The inclusion $\Omega_B\subset\Omega_0$ follows from the fact that $\sigma_B<\sigma_0$ (see~\eqref{eq: sigma}). Moreover, since $R_0R_1>\underline{R}^2>0$, from \eqref{eq: generating function c}
we have, uniformly in $\Omega_B$,
\[\sqrt{R_0^2R_1^2-c^2\tau^2}=R_0R_1+\mathcal O_{C^2}(c),\qquad c\arctan\left(\frac{c\tau}{\sqrt{R_0^2R_1^2-c^2\tau^2}}\right)=\mathcal O_{C^2}(c).\]
Moreover, since $\partial_{12}h_0<0$ on $\Omega_0$ and $\sigma_B<\sigma_0$ we have
\[
\begin{aligned}
\sigma_B^*:&=\max_{t}\partial_1h_c(t,t+\sigma_B)=\max_{t}\partial_1h_0(t,t+\sigma_B)+\mathcal O_{C^1}(c)>\max_{t}\partial_1h_0(t,t+\sigma_0)+\mathcal O_{C^1}(c)\\
&=\sigma_0^*+\mathcal O_{C^1}(c).
\end{aligned}
\]
This implies that for $c\rightarrow 0^+$, $\Sigma_B\subset\Sigma_0$. The relation between $\mathcal P_c$ and $\mathcal P_0$ comes from their definitions in terms of $h_0$ and $h_c$. 
This concludes the proof.
\end{proof}

\section{Chaotic motion}\label{sec: chaos}

To prove Theorems~\ref{thm: chaos FU} and~\ref{thm: chaos} we proceed as follows. In Section~\ref{subsec: Theorem12} we prove 
the existence of chaotic motions for the Fermi-Ulam map $\mathcal P_0$. Then, in Section~\ref{subsec: Theorem13} we see how Lemma \ref{lem: C2 convergence} allows us to extend this result to the breathing circle map $\mathcal P_c$ for sufficiently small $c>0$.

\subsection{The Fermi-Ulam model and proof of Theorem \ref{thm: chaos FU}}\label{subsec: Theorem12}
We assume throughout this section that $R\in\widetilde{\mathcal R}_0$ in the sense of Definition~\ref{def: function R} and we consider the generating function $h_0(t_0,t_1)$ in \eqref{eq: generating function FU} defined on the set $\Omega_0$ in \eqref{Omega0}. 

The argument is structured as follows:

\begin{enumerate}
\item We prove the existence of Aubry-Mather sets for the map $\mathcal P_0$ for every rotation number $\omega\in(1,\sigma_0-1)$. Then, starting from an assumed invariant Lipschitz graph, we prove a bounded-deviation estimate for the lift of the induced homeomorphism $\varphi_0$ (see Lemma \ref{lem: localization0}). As a consequence, all consecutive impact pairs along the graph remain in a compact strip contained in $\Omega_0$.

\item Restricting to the above strip, we obtain uniform bounds depending on $\underline R,\overline R$ and $\kappa_R$ (see~\eqref{eq: kappaR}) that must hold along any invariant Lipschitz graph of rotation number $\omega$. 

\item Relying on these explicit bounds, we show that for $R\in\widetilde{\mathcal R}_0$ (see Definition~\ref{def: function R}) there exists a nonempty set $\Xi_R$ of rotation numbers $\omega\in(1,\sigma_0-1)$ for which the existence of an invariant Lipschitz graph of rotation number $\omega$ is not possible.

\item We choose an irrational $\omega_*\in\Xi_R$ and show that the corresponding Aubry-Mather set is a Cantor set. This yields positive topological entropy for $\mathcal P_0$ by Forni's theorem (see Theorem~\ref{thm: CKAM-forni}), completing the proof of Theorem~\ref{thm: chaos FU}.
\end{enumerate}

\paragraph{Step 1. Aubry-Mather sets and a rotation strip for invariant graphs.}
Since $\sigma_0>2$, for every $\omega\in(1,\sigma_0-1)$ we can choose $0<\beta<\min\{\omega-1,\sigma_0-\omega-1\}$. Consider the compact strip
\begin{equation}\label{eq: Omega0 beta}
\Omega_{0,\beta}:= \left\{(t_0,t_1)\in\mathds R^2:\beta\leq t_1-t_0\leq\sigma_0-\beta\right\}\subset\Omega_0,
\end{equation}
where $\Omega_0$ is defined in~\eqref{Omega0}.

By Lemma~\ref{lem: extension}, there exists a $C^2$ diagonally periodic extension $\widetilde h_0\colon\mathds R^2\to\mathds R$ of $h_0$ in~\eqref{eq: generating function FU} such that $\widetilde h_0=h_0$ on $\Omega_{0,\beta}$, $\partial_{12}\widetilde h_0<0$ on $\mathds R^2$, and $\widetilde h_0$ has quadratic tails. Then the implicit relations
\[\widetilde K_0:=\partial_1\widetilde h_0(t_0,t_1),\qquad\widetilde K_1:=-\partial_2\widetilde h_0(t_0,t_1)\]
define a global exact symplectic twist diffeomorphism
\begin{equation}\label{eq: global twist map0}
\widetilde{\mathcal P}_0\colon\mathds T\times\mathds R\to\mathds T\times\mathds R.
\end{equation}
We will apply Aubry-Mather theory to $\widetilde{\mathcal P}_0$, and then use the bounded-deviation property of minimizers in Lemma~\ref{lem: increment} to conclude that, for $\omega\in(1,\sigma_0-1)$, the minimizing dynamics stays inside $\Omega_{0,\beta}$ and therefore coincides with the dynamics of Fermi-Ulam map $\mathcal P_0$.

\begin{proposition}\label{prop: rotation numbers0}
For every $\omega\in(1,\sigma_0-1)$, the map $\mathcal P_0$ admits:
\begin{itemize}
\item if $\omega=p/q\in\mathds Q$ in lowest terms, a $q$-periodic minimizing orbit $(t_n,K_n)$ with $(t_{n+q},K_{n+q})=(t_n+p,K_n)$ for all $n$;
\item if $\omega\notin\mathds Q$, a compact minimizing invariant set $M_{\omega,0}$ of rotation number $\omega$, contained in a Lipschitz graph over $\mathds T$ and such that the projection to the base is either all of $\mathds T$ (invariant-curve case) or a Cantor subset (Cantor case).
\end{itemize}
\end{proposition}

\begin{proof}[Proof of Proposition~\ref{prop: rotation numbers0}]
Apply Theorem~\ref{thm: Aubry-Mather} to the global twist map $\widetilde{\mathcal P}_0$ generated by $\widetilde h_0$. Minimizing configurations with rotation number $\omega\in(1,\sigma_0-1)$ satisfy the bounded-deviation estimate in Lemma~\ref{lem: increment}, hence remain in $\Omega_{0,\beta}$ defined in~\eqref{eq: Omega0 beta}. Since $\widetilde h_0=h_0$ on $\Omega_{0,\beta}$, these minimizing configurations are also minimizing for $h_0$ and correspond to orbits of the map $\mathcal P_0$ (see Proposition~\ref{prop: DEL orbit correspondence0}).
\end{proof}

For every $\omega\in(1,\sigma_0-1)$, let
\begin{equation}\label{eq: Omega0 omega}
\Omega_\omega:=
\left\{(t_0,t_1)\in\mathds R^2:\omega-1\leq t_1-t_0\leq\omega+1\right\}\subset \Omega_{0,\beta}\subset \Omega_0,
\end{equation}
where $\Omega_0$ is defined in~\eqref{eq: Omega}. A direct consequence of Lemma \ref{lemma: birk} is the following lemma.
\begin{lemma}\label{lem: localization0}
Let $\Gamma_0$ be an invariant Lipschitz graph for the global twist map $\widetilde{\mathcal P}_0$ in~\eqref{eq: global twist map0}, and let $\varphi_0:\mathds R\to\mathds R$ be the corresponding Birkhoff map (see Appendix~\ref{app:CKAM} for the definition of the Birkhoff map). If the rotation number of $\varphi_0$ is $\omega\in(1,\sigma_0-1)$, then for all $x\in\mathds R$, 
\[
(x,\varphi_0(x))\in\Omega_\omega.
\]
With some abuse of notation we will also write $\Gamma_0\subset\Omega_\omega$.
%
%
%
%
\end{lemma}

%
%

\paragraph{Step 2. Converse-KAM bounds and exclusion of invariant graphs.}
Assume, towards a contradiction, that there exists an invariant Lipschitz graph $\Gamma_0$ for $\widetilde{\mathcal P}_0$ of rotation number $\omega\in(1,\sigma_0-1)$. By Lemma~\ref{lem: localization0}, $\Gamma_0\subset\Omega_\omega$. On $\Gamma_0$ we work with the physical generating function $h_0$ in~\eqref{eq: generating function FU}.

Denote by
\[\tau:=t_1-t_0,\quad R_0:=R(t_0),\quad R_1:=R(t_1),\quad S:=R_0+R_1.\]
A direct computation from \eqref{eq: generating function FU} yields in $\Omega_\omega$
\begin{equation}\label{eq: hc11 hc22 hc120}
\begin{aligned}
\partial_{11}h_0(t_0,t_1)&=\frac{\dot R_0^{2}+S\ddot R_0}{\tau}+\frac{2S\dot R_0}{\tau^2}+\frac{S^2}{\tau^3},\\
\partial_{22}h_0(t_0,t_1)&=\frac{\dot R_1^{2}+S\ddot R_1}{\tau}-\frac{2S\dot R_1}{\tau^2}+\frac{S^2}{\tau^3},\\
-\partial_{12}h_0(t_0,t_1)&=\frac{1}{\tau}\left(\frac{S}{\tau}+\dot R_0\right)\left(\frac{S}{\tau}-\dot R_1\right).
\end{aligned}
\end{equation}
Let $\varphi_0$ be the induced homeomorphism on $\Gamma_0$ and define, for $x\in\mathds R$,
\begin{equation}\label{eq: ab0}
a_0(x):=\partial_{22}h_0(\varphi_0^{-1}(x),x)+\partial_{11}h_0(x,\varphi_0(x)),\qquad b_0(x):=-\partial_{12}h_0(\varphi_0^{-1}(x),x).
\end{equation}
By the periodicity of $h_0$ in~\eqref{eq: h0 twist}, both $a_0(x)$ and $b_0(x)$ are $1$-periodic. The twist condition in~\eqref{eq: h0 twist} gives $b_0(x)>0$ for every $x\in\mathds R$.
Let us consider, for $\omega\in(1,\sigma_0-1)$, the continuous functions
\begin{equation}\label{eq: H11 H22 b*0}
\begin{aligned}
H_{11}^0(\omega)&:=\max_{t\in\mathds T,\;\tau\in[\omega-1,\omega+1]}|\partial_{11}h_0(t,t+\tau)|,& H_{22}^0(\omega)&:=\max_{t\in\mathds T,\;\tau\in[\omega-1,\omega+1]}|\partial_{22}h_0(t,t+\tau)|,\\
b_*^0(\omega)&:=\min_{t\in\mathds T,\;\tau\in[\omega-1,\omega+1]}\left(-\partial_{12}h_0(t,t+\tau)\right),& b_0^*(\omega)&:=\max_{t\in\mathds T,\;\tau\in[\omega-1,\omega+1]}\left(-\partial_{12}h_0(t,t+\tau)\right).
\end{aligned}
\end{equation}
We now derive quantitative estimates for the coefficient $a_0(x)$ defined in~\eqref{eq: ab0} along an invariant Lipschitz graph of rotation number $\omega\in(1,\sigma_0-1)$, result of the following lemmas.

\begin{lemma}\label{lem: Alow0}
Let $\Gamma\subset \Omega_\omega$ be an invariant Lipschitz graph of rotation number $\omega\in(1,\sigma_0-1)$, and $a_0(x)$ be defined in~\eqref{eq: ab0}. Then
\begin{equation}\label{eq: Alow 0}
a_0(x)> A_{\mathrm{low}}(\omega,R):=2\frac{\left(b_{\mathrm{low}}(\omega)\right)^2}{H_{\mathrm{up}}(\omega)}\qquad \forall x\in\mathds R,
\end{equation}
where
\[b_{\mathrm{low}}(\omega):=\frac{1}{\omega+1}\left(\frac{2\underline R}{\omega+1}-\|\dot R\|\right)^2,\quad H_{\mathrm{up}}(\omega):=\frac{2\|\dot R\|^2+4\overline R\|\ddot R\|}{\omega-1}+\frac{8\overline R\|\dot R\|}{(\omega-1)^2}+\frac{8\overline R^{\,2}}{(\omega-1)^3}.\]
\end{lemma}

\begin{proof}[Proof of Lemma~\ref{lem: Alow0}]
By Lemma~\ref{lem: localization0}, for every $x\in\mathds R$ the pairs $(\varphi_0^{-1}(x),x)$ and $(x,\varphi_0(x))$ belong to $\Omega_\omega$. Applying~\eqref{eq: ab0} and~\eqref{eq: H11 H22 b*0} we have that $b_0(\cdot)\in[b_*^0(\omega),b_0^*(\omega)]$ and
\[|\partial_{22}h_0(\varphi_0^{-1}(x),x)|\le H_{22}^0(\omega),\qquad|\partial_{11}h_0(x,\varphi_0(x))|\le H_{11}^0(\omega).\]
Hence
\[\frac{a_0(x)}{b_0(\varphi_0(x))}\leq\frac{H_{11}^0(\omega)+H_{22}^0(\omega)}{b_*^0(\omega)},\qquad\frac{a_0(x)}{b_0(x)}\leq\frac{H_{11}^0(\omega)+H_{22}^0(\omega)}{b_*^0(\omega)}.\]
Let
\[B_0(\omega):=\frac{H_{11}^0(\omega)+H_{22}^0(\omega)}{b_*^0(\omega)}>0,\quad C_0(\omega):=\frac{b_*^0(\omega)}{b_0^*(\omega)}\in(0,1],\quad\widetilde C_0(\omega):=\min\left\{C_0(\omega),\frac{B_0(\omega)^2}{8}\right\}>0.\]
Then Theorem~\ref{thm: CKAM-maro} applies with $B^\pm=B_0(\omega)$ and $C^\pm=\widetilde C_0(\omega)$. If we denote by
\[\widetilde D_0^{+}(\omega):=\frac{B_0(\omega)+\sqrt{B_0(\omega)^2-4\widetilde C_0(\omega)}}{2}\]
and
\[\widetilde D_0^{-}(\omega):=\frac{B_0(\omega)-\sqrt{B_0(\omega)^2-4\widetilde C_0(\omega)}}{2\widetilde C_0(\omega)}=\frac{2}{B_0(\omega)+\sqrt{B_0(\omega)^2-4\widetilde C_0(\omega)}},\]
we obtain
\[a_0(x)\geq b_0(\varphi_0(x))\widetilde D_0^-(\omega)+\frac{b_0(x)}{\widetilde D_0^+(\omega)}\qquad\forall x\in\mathds R.\]
Therefore we get, noting that $1/\widetilde D_0^+(\omega)=\widetilde D_0^-(\omega)>1/B_0(\omega)$,
\[a_0(x)\geq b_*^0(\omega)\left(\widetilde D_0^-(\omega)+\frac{1}{\widetilde D_0^+(\omega)}\right)=2b_*^0(\omega)\widetilde D_0^-(\omega)>2\frac{b_*^0(\omega)}{B_0(\omega)}=2\frac{(b_*^0(\omega))^2}{H_{11}^0(\omega)+H_{22}^0(\omega)}.\]
Applying~\eqref{eq: hc11 hc22 hc120},~\eqref{eq: ab0} and~\eqref{eq: H11 H22 b*0}, together with $S\geq2\underline R$ and $\tau\leq\omega+1$, we obtain,
\[b_*^0(\omega)\geq\frac{1}{\omega+1}\left(\frac{2\underline R}{\omega+1}-\|\dot R\|\right)^2=b_{\mathrm{low}}(\omega).\]
Since $\omega\in(1,\sigma_0-1)$, $b_{\mathrm{low}}(\omega)>0$. From~\eqref{eq: hc11 hc22 hc120}, $\tau\geq\omega-1$ and $S\leq2\overline R$. Hence
\[H_{11}^0(\omega)+H_{22}^0(\omega)\leq\frac{2\|\dot R\|^2+4\overline R\|\ddot R\|}{\omega-1}+\frac{8\overline R\|\dot R\|}{(\omega-1)^2}+\frac{8\overline R^2}{(\omega-1)^3}=H_{\mathrm{up}}(\omega),\]
completing the proof.
\end{proof}

\begin{lemma}\label{lem: Aup0}
Let $\Gamma\subset \Omega_\omega$ be an invariant Lipschitz graph of rotation number $\omega\in(1,\sigma_0-1)$, and let $a_0(x)$ be defined in~\eqref{eq: ab0}. Then there exists $x_0\in \mathds R$ such that
\begin{equation}\label{eq: Aup 0}
a_0(x_0)< A_{\mathrm{up}}(\omega,R):=-\frac{2\kappa_R(\underline R+\overline R)}{\omega+1}+\frac{8\overline R^{2}}{(\omega-1)^3},
\end{equation}
where $\kappa_R$ is defined in~\eqref{eq: kappaR}.
\end{lemma}

\begin{proof}[Proof of Lemma~\ref{lem: Aup0}]
By Definition~\ref{def: function R}, there exists $\bar t\in[0,1)$ such that
\[R(\bar t)=\overline R,\qquad \ddot R(\bar t)<-\kappa_R.\]
%
Set $x_0\equiv\bar t\ (\mathrm{mod}\ 1)$ and let $\varphi_0$ be the lift of the homeomorphism induced by $\Gamma_0$ on the base. Define $t_-:=\varphi_0^{-1}(x_0)$, $t_+:=\varphi_0(x_0)$ and
\[\tau_-:=x_0-t_->0,\qquad\tau_+:=t_+-x_0>0,\qquad S_-:=R(t_-)+R(x_0),\qquad S_+:=R(x_0)+R(t_+).\]
By Lemma~\ref{lem: localization0}, $\tau_\pm\in[\omega-1,\omega+1]$ and $S_\pm\in[\overline R+\underline R,\,2\overline R]$. Using~\eqref{eq: hc11 hc22 hc120} at $x_0=\bar t$ and $\dot R(\bar t)=0$, we obtain
\[a_0(x_0)=\ddot R(\bar t)\left(\frac{S_-}{\tau_-}+\frac{S_+}{\tau_+}\right)+\frac{S_-^2}{\tau_-^3}+\frac{S_+^2}{\tau_+^3}.\]
Since $\ddot R(\bar t)<0$, $S_\pm\geq\overline R+\underline R$ and $\tau_\pm\leq\omega+1$,
\[\ddot R(\bar t)\left(\frac{S_-}{\tau_-}+\frac{S_+}{\tau_+}\right)\leq2\ddot R(\bar t)\frac{\overline R+\underline R}{\omega+1}.\]
Moreover $S_\pm\leq2\overline R$ and $\tau_\pm\geq\omega-1$, hence
\[\frac{S_-^2}{\tau_-^3}+\frac{S_+^2}{\tau_+^3}\leq2\frac{(2\overline R)^2}{(\omega-1)^3}.\]
Substituting~\eqref{eq: kappaR} we obtain~\eqref{eq: Aup 0}, completing the proof.
\end{proof}

\paragraph{Step 3. A nontrivial set of rotation numbers with no invariant graphs.}

Define
\begin{equation}\label{eq: Xi_R0}
\Xi_R:=\left\{\omega\in(1,\sigma_0-1):A_{\mathrm{up}}(\omega,R)<A_{\mathrm{low}}(\omega,R)\right\},
\end{equation}
where $A_{\mathrm{low}}(\omega,R)$ and $A_{\mathrm{up}}(\omega,R)$ are defined in~\eqref{eq: Alow 0} and~\eqref{eq: Aup 0}, respectively.

\begin{proposition}\label{prop: no invariant Lipschitz0}
If $\omega\in\Xi_R$, the map $\mathcal P_0$ admits no invariant Lipschitz graph of rotation number $\omega$.
\end{proposition}

\begin{proof}[Proof of Proposition~\ref{prop: no invariant Lipschitz0}]
Assume by contradiction that $\mathcal P_0$ admits an invariant Lipschitz graph $\Gamma_0$ of rotation number $\omega\in(1,\sigma_0-1)$. Lemma~\ref{lem: Aup0} gives a point $x_0\in\mathds R$ such that $a_0(x_0)\leq A_{\mathrm{up}}(\omega,R)$, while Lemma~\ref{lem: Alow0} gives $a_0(x)\geq A_{\mathrm{low}}(\omega,R)$ for all $x\in\mathds R$. Since $\omega\in \Xi_R$, these inequalities are incompatible.
\end{proof}

\begin{lemma}\label{lem: XiR nonempty0}
Let $R\in\widetilde{\mathcal R}_0$. Then the set $\Xi_R$ is not empty and contains an open interval $\mathcal I\subset(1,\sigma_0-1)$.
\end{lemma}

\begin{proof}[Proof of Lemma~\ref{lem: XiR nonempty0}]
Relying on~\eqref{eq: Alow 0} and~\eqref{eq: Aup 0}, it suffices to prove that there exists $\mathcal I\subset\Xi_R$ such that for all $\omega\in\mathcal I$ we have
\[\kappa_R>F(\omega),\]
where
\begin{equation}\label{eq: Fw}
F(\omega):=\frac{\omega+1}{2(\overline R+\underline R)}
\left(\frac{8\overline R^2}{(\omega-1)^3}-A_{\mathrm{low}}(\omega,R)\right),
\end{equation}
and $A_{\mathrm{low}}(\omega,R)$ is defined in~\eqref{eq: Alow 0}.

We first note that $F(\omega)>0$ for every $\omega\in(1,\sigma_0-1)$. Indeed, by~\eqref{eq: Alow 0},
\[A_{\mathrm{low}}(\omega,R)=\frac{2b_{\mathrm{low}}(\omega)^2}{H_{\mathrm{up}}(\omega)},\qquad b_{\mathrm{low}}(\omega)=\frac{1}{\omega+1}\left(\frac{2\underline R}{\omega+1}-\|\dot R\|\right)^2.\]
Since $\omega+1<\sigma_0=\underline R/\|\dot R\|$ (see~\eqref{eq: sigma}), we have $\frac{2\underline R}{\omega+1}-\|\dot R\|>0$, and hence $b_{\mathrm{low}}(\omega)>0$. Moreover,
\[b_{\mathrm{low}}(\omega)<\frac{1}{\omega+1}\left(\frac{2\underline R}{\omega+1}\right)^2=\frac{4\underline R^2}{(\omega+1)^3}\leq\frac{4\overline R^2}{(\omega+1)^3}.\]
On the other hand, by the definition of $H_{\mathrm{up}}(\omega)$,
\[H_{\mathrm{up}}(\omega)=\frac{2\|\dot R\|^2+4\overline R\|\ddot R\|}{\omega-1}+\frac{8\overline R\|\dot R\|}{(\omega-1)^2}+\frac{8\overline R^2}{(\omega-1)^3}>\frac{8\overline R^2}{(\omega-1)^3}.\]
Hence
\[A_{\mathrm{low}}(\omega,R)<\frac{2\cdot16\overline R^4}{(\omega+1)^6}\cdot\frac{(\omega-1)^3}{8\overline R^2}=\frac{4\overline R^2(\omega-1)^3}{(\omega+1)^6}.\]
Since
\[\frac{8\overline R^2}{(\omega-1)^3}>\frac{4\overline R^2(\omega-1)^3}{(\omega+1)^6},\]
it follows that $F(\omega)>0$ for every $\omega\in(1,\sigma_0-1)$.

Set
\[\alpha:=\inf_{\omega\in(1,\sigma_0-1)}F(\omega).\]
By~\eqref{eq: sharpened criterion}, we have $\kappa_R>\alpha$. Let $\delta:=\frac{\kappa_R-\alpha}{2}>0$. By the definition of infimum, there exists $\omega_0\in(1,\sigma_0-1)$ such that
\[F(\omega_0)<\alpha+\delta=\frac{\kappa_R+\alpha}{2}<\kappa_R.\]
Then, by continuity of $F$ at $\omega_0$, there exists $\varepsilon>0$ such that for every $\omega\in(\omega_0-\varepsilon,\omega_0+\varepsilon)$,
\[F(\omega)< F(\omega_0) + \delta <\kappa_R.\]
Shrinking $\varepsilon$ if necessary, we may assume
\[\mathcal I:=(\omega_0-\varepsilon,\omega_0+\varepsilon)\subset(1,\sigma_0-1).\]
Hence $\mathcal I\subset\Xi_R$, completing the proof.
\end{proof}

\paragraph{Step 4. Proof of Theorem~\ref{thm: chaos FU}.}
By Lemma~\ref{lem: XiR nonempty0}, the set $\Xi_R$ in~\eqref{eq: Xi_R0} contains a nonempty open interval $\mathcal I\subset(1,\sigma_0-1)$. Choose $\omega_*\in\mathcal I\cap(\mathds R\setminus\mathds Q).$ By Propostion \ref{prop: rotation numbers0} there exists a compact minimizing $\mathcal P_0$-invariant set $M_{\omega_*,0}$ with rotation number $\omega_*$ which, by Proposition~\ref{prop: no invariant Lipschitz0}, is a Cantor set.

Therefore Theorem~\ref{thm: CKAM-forni} yields an invariant ergodic probability measure $\mu_{\omega_*}$ for $\mathcal P_0$ with rotation number $\omega_*$ and positive metric entropy. In particular, $\mathcal P_0$ has positive topological entropy.

This completes the proof of Theorem~\ref{thm: chaos FU}.

\subsection{The breathing circle billiard and proof of Theorem~\ref{thm: chaos}}\label{subsec: Theorem13}

We now prove Theorem~\ref{thm: chaos}. The argument follows the same structure as the proof of Theorem~\ref{thm: chaos FU}, with the generating function $h_c$ treated as a small perturbation of $h_0$ through Lemma~\ref{lem: C2 convergence}.

Throughout this section we assume that $R\in\widetilde{\mathcal R}_B$  in the sense of Definition~\ref{def: function R} and we consider the generating function $h_c$ defined in~\eqref{eq: generating function c} on the strip
$\Omega_B$ defined in~\eqref{eq: Omega}.

Fix \(\omega\in(1,\sigma_B-1)\) with $\sigma_B$ defined in~\eqref{eq: sigma}, choose $0<\beta<\min\{\omega-1,\sigma_B-\omega-1\}$
and consider the compact strip
\begin{equation}\label{eq: Omega Bbeta}
\Omega_{B,\beta}:=
\{(t_0,t_1)\in\mathds R^2:\beta\le t_1-t_0\le \sigma_B-\beta\}
\subset\Omega_B.
\end{equation}
Extending $h_c$ outside $\Omega_{B,\beta}$ via Lemma~\ref{lem: extension} leads to the extended map $\widetilde {\mathcal P}_c$ of $\mathcal P_c$ defined in~\eqref{eq: billiard map}. The following proposition, whose proof is completely analogous to the one in Proposition~\ref{prop: rotation numbers0}, provides the Aubry-Mather sets available for $\mathcal P_c$.

\begin{proposition}
\label{prop: rotation number c}
Fix $\varepsilon \in(0,1)$. For every $\omega\in(1,\sigma_B-1)$ and $c\in \left(0,\varepsilon\frac{\underline R^2}{\sigma_B}\right)$ the map $\mathcal P_c$ in~\eqref{eq: billiard map} admits:
\begin{itemize}
\item if $\omega=p/q\in\mathds Q$ in lowest terms, a $q$-periodic
minimizing orbit $(t_n,K_n)$ with $(t_{n+q},K_{n+q})=(t_n+p,K_n)$ for all $n\in\mathds Z$;
\item if $\omega\notin\mathds Q$, a compact minimizing invariant set $M_{\omega,c}$ of rotation number $\omega$, contained in a Lipschitz graph over $\mathds T$, whose projection to the base is either all of $\mathds T$ (invariant-curve case) or a Cantor subset (Cantor case).
\end{itemize}
\end{proposition}




For every $\omega\in(1,\sigma_B-1)$ let
\begin{equation}\label{eq: Omega-omega-B}
\Omega_\omega^B:=
\{(t_0,t_1)\in\mathds R^2:
\omega-1\le t_1-t_0\le \omega+1\}
\subset\Omega_{B,\beta}\subset\Omega_B .
\end{equation}
The next lemma is the analogue of Lemma~\ref{lem: localization0}, and also follows as a consequence of Lemma~\ref{lemma: birk}.

\begin{lemma}\label{lem: localization}
Let $\Gamma_c$ be an invariant Lipschitz graph for the global twist map $\widetilde{\mathcal P}_c$, and let
$\varphi_c:\mathds R\to\mathds R$ be the corresponding Birkhoff map. If the rotation number of $\varphi_c$ is $\omega\in(1,\sigma_B-1)$, then
\[(x,\varphi_c(x))\in\Omega_\omega^B
\qquad \forall x\in\mathds R .\]
With some abuse of notation we will also write $\Gamma_c\subset\Omega_\omega^B$.
\end{lemma}

Assume now that $\mathcal P_c$ has an invariant Lipschitz graph $\Gamma_c$ with rotation number $\omega\in(1,\sigma_B-1)$. Let $\varphi_c$ be its Birkhoff map. We define
\begin{equation}\label{eq:abc}
a_c(x):=\partial_{22}h_c(\varphi_c^{-1}(x),x)+\partial_{11}h_c(x,\varphi_c(x)),\qquad b_c(x):=-\partial_{12}h_c(\varphi_c^{-1}(x),x).
\end{equation}
By the twist condition and the diagonal periodicity of $h_c$ in~\eqref{eq: diagonal periodic h}, we have that $b_c(x)>0$ and both $a_c$ and $b_c$ are $1$-periodic.

The following two lemmas are the perturbative counterparts of Lemmas~\ref{lem: Alow0} and~\ref{lem: Aup0}, respectively.

\begin{lemma}
\label{lem: Alowc}
Let $\Gamma_c\subset\Omega_\omega^B$ be an invariant Lipschitz graph of rotation number $\omega\in(1,\sigma_B-1)$, and let $a_c$ be defined by \eqref{eq:abc}. Then, uniformly in $x\in \mathds R$ (but not in $\omega$)
\begin{equation}\label{eq: Alowc}
a_c(x)> A_{\mathrm{low}}(\omega,R) + \mathcal O(c),
\end{equation}
where $A_{\mathrm{low}}$ is the function defined in~\eqref{eq: Alow 0} restricted to $\omega\in(1,\sigma_B-1)$.
\end{lemma}

\begin{proof}[Proof of Lemma~\ref{lem: Alowc}]
By Lemma \ref{lem: localization}, the pairs $(\varphi_c^{-1}(x),x)$ and $(x,\varphi_c(x))$ belong to $\Omega_\omega^B$ for
every $x\in\mathds R$. On this compact strip, Lemma~\ref{lem: C2 convergence} gives
\[\partial_{ij}h_c=\partial_{ij}h_0+\mathcal O(c),\qquad i,j\in\{1,2\},\]
uniformly as $c\to0^+$.

Therefore the quantities in~\eqref{eq: H11 H22 b*0} (now restricted to $\omega\in(1,\sigma_B-1)$) are perturbed by $\mathcal O(c)$ when $h_0$ is replaced by $h_c$. Applying Theorem~\ref{thm: CKAM-maro} exactly as in the proof of Lemma~\ref{lem: Alow0} then gives the lower bound \eqref{eq: Alowc} and completes the proof.
\end{proof}

\begin{lemma}
\label{lem: Aupc}
Let $\Gamma_c\subset\Omega_\omega^B$ be an invariant Lipschitz graph of rotation number $\omega\in(1,\sigma_B-1)$, and let $a_c$ be defined by \eqref{eq:abc}. Then there exist $x_0\in\mathds R$  such that
\begin{equation}\label{eq: Aupc}
a_c(x_0)< A_{\mathrm{up}}(\omega,R)+ \mathcal O(c),
\end{equation}
where $A_{\mathrm{up}}$ is the function defined in \eqref{eq: Aup 0}, restricted to $\omega\in(1,\sigma_B-1)$ and the infimum in the definition of $\kappa_R$ in~\eqref{eq: kappaR} is taken over $\omega\in(1,\sigma_B-1)$.
\end{lemma}

\begin{proof}[Proof of Lemma~\ref{lem: Aupc}]
Choosing $x_0 \equiv \bar t \mod 1$ where $\bar t$ is defined in~\eqref{eq: kappaR} implies that both $(x_0,\varphi_c(x_0))$ and $(\varphi_c^{-1}(x_0),x_0)$ belong to $\Omega_\omega^B$. The rest of the proof is completely analogous to the one of Lemma~\ref{lem: Aup0}.
\end{proof}

We now define the set of rotation numbers for which the two previous estimates are incompatible:
\begin{equation}\label{eq:XiRc}
\Xi_R^B:=\left\{\omega\in(1,\sigma_B-1):A_{\mathrm{up}}(\omega,R)<A_{\mathrm{low}}(\omega,R)\right\}\subset \Xi_R,
\end{equation}
where $A_{\mathrm{low}}(\omega,R)$, $A_\mathrm{up}(\omega,R)$ and $\Xi_R$ are defined in~\eqref{eq: Alow 0},~\eqref{eq: Aup 0} and~\eqref{eq: Xi_R0}, respectively.
\begin{proposition}\label{prop: no invariant Lipschitzc}
If $\omega\in\Xi_R^B$, then there exists $c_\omega>0$ such that, for every $c\in(0,c_\omega)$, the map $\mathcal P_c$ admits no invariant Lipschitz graph with rotation number $\omega$.
\end{proposition}

\begin{proof}[Proof of Proposition~\ref{prop: no invariant Lipschitzc}]
Fix $\omega\in\Xi_R^B$ and suppose by contradiction that for every $\varepsilon>0$ there exist $c\in(0,\varepsilon)$ such that the map $\mathcal P_c$ admits an invariant Lipschitz graph $\Gamma_c$ of rotation number $\omega$. By Lemma~\ref{lem: Alowc} for all $x\in\mathds R$ we have $a_c(x)> A_{\mathrm{low}}(\omega,R)+\mathcal{O}(c)$. On the other hand, Lemma~\ref{lem: Aupc} gives a point $x_0\in\mathds R$ such that $a_c(x_0)< A_{\mathrm{up}}(\omega,R)+\mathcal{O}(c)$. For $\varepsilon\rightarrow 0$ this is only possible if   
$A_{\mathrm{up}}(\omega,R)>A_{\mathrm{low}}(\omega,R)$, in contradiction whit the fact that $\omega\in\Xi_R^B$.
\end{proof}

\begin{lemma}
\label{lem: XiR nonempty}
Let \(R\in\widetilde{\mathcal R}_B\). Then the set \(\Xi_R^B\) is not empty and contains an open
interval \(\mathcal I_B\subset(1,\sigma_B-1)\).
\end{lemma}

\begin{proof}[Proof of Lemma~\ref{lem: XiR nonempty}]

We argue as in the proof of Lemma~\ref{lem: XiR nonempty0}, restricting the function $F$ defined in~\eqref{eq: Fw} to the interval $(1,\sigma_B-1)$. Since $\sigma_B\le \sigma_0$, this interval is contained in $(1,\sigma_0-1)$, and hence the positivity of $F$ proved in
Lemma~\ref{lem: XiR nonempty0} still holds here.

Because \(R\in\widetilde{\mathcal R}_B\), condition~\eqref{eq: kappaR} gives
\[\kappa_R>\inf_{\omega\in(1,\sigma_B-1)}F(\omega).\]
Thus there exists $\omega_0\in(1,\sigma_B-1)$ such that $F(\omega_0)<\kappa_R$. By continuity of $F$, this inequality holds on a nonempty open interval $I_B\subset(1,\sigma_B-1)$. Therefore $\mathcal I_B\subset\Xi_R^B$, completing the proof.
\end{proof}

We can now conclude the proof of Theorem~\ref{thm: chaos}.

\begin{proof}[Proof of Theorem~\ref{thm: chaos}]
By Lemma~\ref{lem: XiR nonempty}, the set $\Xi_R^B$ contains a nonempty open interval $\mathcal I_B\subset(1,\sigma_B-1)$. Choose $\omega_*\in\mathcal I_B\cap(\mathds R\setminus\mathds Q)$. By Proposition \ref{prop: rotation number c}, there exists a compact minimizing $\mathcal P_c$-invariant set $M_{\omega_*,c}$ with rotation number $\omega_*$. Since $\omega_*$ is irrational, $M_{\omega_*,c}$ is either contained in an invariant Lipschitz graph or is a Cantor set. Since Proposition~\ref{prop: no invariant Lipschitzc} applies to $\omega_*$, there exist $c_0>0$ such that for every $c\in(0,c_0)$, the set $M_{\omega_*,c}$ is a Cantor set.  

Therefore Theorem~\ref{thm: CKAM-forni} yields an invariant ergodic probability measure $\mu_{\omega_*}$ for $\widetilde{\mathcal P}_c$ with rotation number $\omega_*$ and positive metric entropy. In particular, $\mathcal P_c$ has positive topological entropy for every $c\in(0,c_0)$.

This completes the proof of Theorem~\ref{thm: chaos}.
\end{proof}

\appendix
\section{Aubry-Mather theory for exact symplectic twist maps}\label{app:Aubry-Mather}

We recall the variational formulation for exact symplectic twist maps via generating functions and the basic structure of minimizing invariant sets. We refer to \cite{Bangert1988,MatherForni1994} for classical sources and to the appendix of \cite{BonannoMaro2022} for a presentation close to the impact setting.

\begin{lemma}\label{lem: extension}
Let $\Omega \subset \mathds{R}^2$ be an open strip of the form
\[\Omega = \{(t_0,t_1) \in \mathds{R}^2 : a < t_1 - t_0 < b\},\]
with $a < b$. Let $h : \Omega \to \mathds{R}$ be a $C^2$ function such that
\begin{itemize}
\item[\textup{(i)}] $h(t_0+1,t_1+1) = h(t_0,t_1)$ for all $(t_0,t_1)\in\Omega$;
\item[\textup{(ii)}] $\partial_{t_0 t_1} h(t_0,t_1) < 0$ for all $(t_0,t_1)\in\Omega$.
\end{itemize}
Let $\Omega' = \{(t_0,t_1) : a' < t_1-t_0 < b'\}$ with $a<a'<b'<b$, and choose $a'',b''$ with $a<a''<a'<b'<b''<b$. Then there exists a $C^2$ function $\widetilde h:\mathds R^2\to\mathds R$ such that:
\begin{enumerate}
\item $\widetilde h(t_0+1,t_1+1)=\widetilde h(t_0,t_1)$ for all $(t_0,t_1)\in\mathds R^2$;
\item $\partial_{t_0t_1}\widetilde h(t_0,t_1)<0$ for all $(t_0,t_1)\in\mathds R^2$;
\item $\widetilde h=h$ on $\Omega'$;
\item there exists $\lambda>0$ such that $\widetilde h(t_0,t_1)=\frac{\lambda}{2}(t_1-t_0)^2$ whenever $t_1-t_0\notin(a'',b'')$.
\end{enumerate}
In particular, the implicit relations
\[y_0=\partial_1\widetilde h(t_0,t_1),\qquad y_1=-\partial_2\widetilde h(t_0,t_1)\]
define a global exact symplectic twist diffeomorphism on $\mathds T\times\mathds R$.
\end{lemma}

The following lemma is the standard bounded-deviation estimate for minimizing configurations in the Aubry-Mather theory (see \cite{Bangert1988,MatherForni1994}).

\begin{lemma}\label{lem: increment}
Let $(t_n)_{n \in \mathds{Z}}$ be a minimizing configuration for $\widetilde h$ with rotation number $\omega$. Then
\[\left| t_{n+1} - t_n - \omega \right| \leq 1 \qquad \forall n \in \mathds{Z}.\]
\end{lemma}

\begin{theorem}[Aubry-Mather~\cite{Bangert1988}]\label{thm: Aubry-Mather}
Let $T$ be an exact symplectic twist map. For every $\omega\in\mathds R$ there exists a nonempty compact $T$-invariant set $M_\omega$ formed by minimizing orbits with rotation number $\omega$. 
\begin{itemize}
\item If $\omega=p/q$ is rational (in lowest terms), $M_\omega$ contains a $q$-periodic minimizing orbit.

\item If $\omega$ is irrational, then $M_\omega$ is an ordered set and is contained in a Lipschitz graph over $\mathds T$. More precisely, its projection to the base is either all of $\mathds T$ (invariant curve case) or a Cantor subset of $\mathds T$ (Cantor case).
\end{itemize}
\end{theorem}

\section{Converse-KAM theory}\label{app:CKAM}

We collect the converse-KAM results used in Section~\ref{sec: chaos}.

Let $F$ be an exact symplectic twist map generated by a $C^2$ function $h$ with $h_{12} < 0$. Let $\Gamma$ be an invariant Lipschitz graph of $F$. The restriction of $F$ to $\Gamma$, projected onto the base, defines an orientation-preserving circle homeomorphism. We denote by $\varphi\colon \mathds R \to \mathds R$ a lift of this map. Thus, if $(x_n,y_n)$ is an orbit on $\Gamma$, then its base coordinates satisfy $x_{n+1} = \varphi(x_n)$ (the map $\varphi$ is also known as the Birkhoff map associated to $\Gamma$). We recall the following results.

\begin{lemma}\label{lemma: birk}
    For every invariant curve $\Gamma$ of $F$ there exists an increasing bi-Lipschitz homeomorphism $\varphi:\mathds R\rightarrow\mathds R$ such that $\varphi(x+1)=\varphi(x)+1$ and 
  \begin{equation}
    \label{dEL}
  h_2(\varphi^{-1}(x),x)+h_1(x,\varphi(x)) =0,
  \end{equation}
for every $x\in\mathds R$. Moreover, if $\varphi$ has rotation number $\omega$, then for all $n\in\mathds Z$ and all $x\in\mathds R$,
\begin{equation}\label{eq: bounded deviation varphi}
\left|\varphi^n(x)-x-n\omega\right|\leq1.
\end{equation}
\end{lemma}

Denote by
\[a(x):=h_{22}(\varphi^{-1}(x),x)+h_{11}(x,\varphi(x)),\qquad b(x):=-h_{12}(\varphi^{-1}(x),x).\]
Then $b(x)>0$ by the twist property.

\begin{theorem}[\cite{Maro2020}]\label{thm: CKAM-maro}
Let $\Gamma$ be an invariant Lipschitz graph for an exact symplectic twist map with $C^2$ generating function $h$ and $h_{12}<0$, and let $a,b$ be defined as above. Let $\varphi$ be the corresponding Birkhoff map on $\Gamma$. Assume that there exist constants $B^\pm>0$ and $C^\pm>0$ such that
\begin{equation*}
\begin{aligned}
B^+\geq \sup_{x\in\mathds R}\frac{a(x)}{b(\varphi(x))},\quad B^-\geq \sup_{x\in\mathds R}\frac{a(x)}{b(x)},\quad C^+\leq \inf_{x\in\mathds R}\frac{b(\varphi(x))}{b(x)},\quad C^-\leq \inf_{x\in\mathds R}\frac{b(x)}{b(\varphi(x))},
\end{aligned}
\end{equation*}
and that $(B^\pm)^2-4C^\pm > 0$. Define
\[D^-:=\frac{B^- - \sqrt{(B^-)^2-4C^-}}{2C^-},\qquad D^+:=\frac{B^+ + \sqrt{(B^+)^2-4C^+}}{2}.\]
Then for every $x\in\mathds R$ one has
\[a(x) \ge b(\varphi(x))D^- + \frac{b(x)}{D^+}.\]
\end{theorem}

\begin{theorem}[Forni~\cite{Forni1996}]\label{thm: CKAM-forni}
Let $T$ be an exact symplectic twist map and let $\omega\notin\mathds Q$. Assume that the Mather set $\mathcal M_\omega$ is not contained in an invariant continuous graph. Then there exists an invariant ergodic probability measure
$\mu$ with rotation number $\omega$ and positive metric entropy. In particular, $T$ has positive topological entropy.
\end{theorem}

\printbibliography
\end{document}